\pdfoutput=1 
\documentclass[11pt]{amsart}
\usepackage[utf8]{inputenc}
\usepackage{amssymb,amsmath,amsthm,enumerate,enumitem,colonequals,mlmodern,tikz-cd,microtype,stmaryrd}
\usepackage[cal=euler,bfcal,bb=px,bfbb]{mathalpha}
\usepackage[top=3.75cm, bottom=3cm, left=3.5cm, right=3.5cm]{geometry}

\makeatletter
\@namedef{subjclassname@2020}{\textup{2020} Mathematics Subject Classification}
\makeatother

\usepackage{xcolor}
\colorlet{darkblue}{blue!55!black}
\colorlet{darkcyan}{cyan!50!black}
\colorlet{darkgreen}{green!60!black}

\usepackage{hyperref}
\hypersetup{
    colorlinks=true,
    linkcolor= darkblue,
    urlcolor= darkcyan,
    citecolor= darkgreen,
}

\def\eqref#1{\textcolor{darkblue}{(\ref{#1})}}

\PassOptionsToPackage{hyphens}{url}
\usepackage{hyperref}

\usepackage[nameinlink]{cleveref} 
\Crefformat{section}{#2\S#1#3} 
\Crefmultiformat{section}{#2\S\S#1#3}{ and~#2#1#3}{, #2#1#3}{, and~#2#1#3}
\crefname{hypothesis}{hypothesis}{hypotheses}
\Crefname{hypothesis}{Hypothesis}{Hypotheses}

\usepackage[pagewise]{lineno}
\let\oldequation\equation
\let\oldendequation\endequation
\renewenvironment{equation}{\linenomathNonumbers\oldequation}{\oldendequation\endlinenomath}
\expandafter\let\expandafter\oldequationstar\csname equation*\endcsname
\expandafter\let\expandafter\oldendequationstar\csname endequation*\endcsname
\renewenvironment{equation*}{\linenomathNonumbers\oldequationstar}{\oldendequationstar\endlinenomath}
\let\oldalign\align
\let\oldendalign\endalign
\renewenvironment{align}{\linenomathNonumbers\oldalign}{\oldendalign\endlinenomath}
\expandafter\let\expandafter\oldalignstar\csname align*\endcsname
\expandafter\let\expandafter\oldendalignstar\csname endalign*\endcsname
\renewenvironment{align*}{\linenomathNonumbers\oldalignstar}{\oldendalignstar\endlinenomath}


\newcounter{intro}
\newcounter{HypCounter}

\newtheorem{introthm}[intro]{Theorem}

\newtheorem{introcor}[intro]{Corollary}

\newtheorem{introprop}[intro]{Proposition}

\theoremstyle{plain}
\newtheorem{theorem}{Theorem}[section]
\newtheorem{lemma}[theorem]{Lemma}
\newtheorem{corollary}[theorem]{Corollary}
\newtheorem{proposition}[theorem]{Proposition}

\theoremstyle{definition}

\newtheorem{definition}[theorem]{Definition}
\newtheorem{example}[theorem]{Example}

\newtheorem*{hypothesis*}{Hypothesis}

\newtheorem{remark}[theorem]{Remark}

\newtheorem*{ack}{Acknowledgements}
\newtheorem*{notation}{Notation}

\numberwithin{equation}{section}
\numberwithin{theorem}{section}

\title[Strong generation for module categories]{Strong generation for module categories}

\author[S.~Dey]{Souvik Dey}
\address{S.~Dey,
Faculty of Mathematics and Physics,
Department of Algebra,
Charles University, 
Sokolovsk\'{a} 83, 186 75 Praha, 
Czech Republic, \url{https://orcid.org/0000-0001-8265-3301}}
\email{souvik.dey@matfyz.cuni.cz}

\author[P.~Lank]{Pat Lank}
\address{P.~Lank,
Department of Mathematics,
University of South Carolina, 
Columbia, SC 29208,
U.S.A.}
\email{plankmathematics@gmail.com}

\author[R.~Takahashi]{Ryo Takahashi}
\address{R.~Takahashi,
Graduate School of Mathematics,
Nagoya University,
\break Furocho, Chikusaku, Nagoya 464-8602, Japan}
\email{takahashi@math.nagoya-u.ac.jp}

\date{\today}



\begin{document}

\begin{abstract}
    This article investigates strong generation within the module category of a commutative Noetherian ring. We establish a criterion for such rings to possess strong generators within their module category, addressing a question raised by Iyengar and Takahashi. As a consequence, this not only demonstrates that any Noetherian quasi-excellent ring of finite Krull dimension satisfies this criterion, but applies to rings outside this class. Additionally, we identify explicit strong generators within the module category for rings of prime characteristic, and establish upper bounds on Rouquier dimension in terms of classical numerical invariants for modules.
\end{abstract}

\maketitle

\section{Introduction}
\label{sec:intro}

This work is concerned with the existence of strong generators in the category
of finitely generated modules over a commutative Noetherian ring. Our main result establishes a useful sufficiency criterion for detecting strong generators in module categories and their associated bounded derived categories, as well as for the vanishing of cohomology annihilators. Consequently, we make further strides towards representation theory of commutative rings in studying the structure of their module categories.

Recall a concept of generation in a triangulated category, as introduced by \cite{Bondal/VandenBergh:2003}, which we'll apply to categories constructed from modules over a ring. Let $R$ be a Noetherian ring, $\operatorname{mod}R$ its category of finitely generated modules, and $D^b(\operatorname{mod}R)$ its associated bounded derived category. For an object $G$ in $D^b (\operatorname{mod}R)$, $\operatorname{thick}^n (G)$ is the smallest full subcategory generated by $G$ using only shifts, retracts of direct sums, and at most $n$ cones. An object $G$ in $D^b(\operatorname{mod}R)$ is said to be a \textit{strong generator} if $\operatorname{thick}^n (G) = D^b(\operatorname{mod}R)$ for some $n$, and the \textit{Rouquier dimension} of $D^b( \operatorname{mod}R)$ is the smallest such integer $n$. See \Cref{sec:generation} for details.

There has been significant activity towards understanding which Noetherian rings $R$ that enjoy property that $D^b(\operatorname{mod}R)$ admits a strong generator \cite{Rouquier:2008, Iyengar/Takahashi:2016, Aoki:2021}, as well as explicitly identifying such objects \cite{BILMP:2023}. For instance, if $R$ is a quasi-excellent Noetherian ring of finite Krull dimension (i.e., an algebra of finite type over a field), then $D^b(\operatorname{mod}R)$ admits a strong generator \cite{Aoki:2021}.

We take a different approach towards studying the existence of strong generators in the bounded derived category by working with a concept introduced for module categories. Let $R$ be a Noetherian ring. An object $G$ in $\operatorname{mod}R$ is said to be a \textit{strong generator} if there exist non-negative integers $n$ and $s$ such that the $s$-th syzygy $\Omega^s_R(M)$ of any finitely generated $R$-module $M$ can be obtained from $G$ using at most $n$ extensions.

This concept was introduced in \cite{Iyengar/Takahashi:2016} and was shown to be closely related to the vanishing of the cohomology annihilator ideal $\operatorname{ca}(R)$ of $R$. That is, the existence of ring elements $r$ in $R$ for which there exists an integer $n$ such that $r$ annihilates $\operatorname{Ext}^n_R(A,B)$ for all objects $A,B$ in $\operatorname{mod}R$ is linked to the existence of strong generators in $\operatorname{mod}R$. The following result makes further progress in this direction.

\begin{introthm}\label{thm:strong_generation}
    For any Noetherian ring $R$, the following conditions are equivalent:
    \begin{enumerate}
        \item $\operatorname{ca}(R/\mathfrak{p})\neq 0$ for every prime ideal $\mathfrak{p}$ in $R$
        \item $\operatorname{mod}R/\mathfrak{p}$ admits a strong generator for every prime ideal $\mathfrak{p}$ of $R$
        \item $D^b (\operatorname{mod}R/\mathfrak{p})$ admits a strong generator for every prime ideal $\mathfrak{p}$ of $R$.
    \end{enumerate}
    If any one of these conditions is satisfied, then $R$ has finite Krull dimension, and both $\operatorname{mod}R/I$ and $D^b (\operatorname{mod}R/I)$ admit strong generators where $I$ is any ideal of $R$.
\end{introthm}

\Cref{thm:strong_generation}, coupled with the main result of \cite{Aoki:2021}, asserts that $\operatorname{mod}R$ admits a strong generator if $R$ is a quasi-excellent Noetherian ring of finite Krull dimension (see \Cref{cor:quasi_excellent_strong_gen}). This affirmatively answers \cite[Question 5.5]{Iyengar/Takahashi:2016}. As an interesting consequence of this, we observe that if $R$ is a local  ring whose completion is an isolated singularity, then $\operatorname{mod } R$ admits a strong generator. See \Cref{prop:cmiso} for details. 
However, these are not the only applications of \Cref{thm:strong_generation}. For example, although any discrete valuation ring satisfies the conditions of \Cref{thm:strong_generation}, they are not always quasi-excellent (e.g.\ \cite[Example 3.8]{Dey/Lank:2024}). Hence, \Cref{thm:strong_generation} offers a new approach to detecting the finiteness for Rouquier dimension of $D^b (\operatorname{mod}R)$ by simplifying the problem into a strictly module-theoretic context, thereby opening opportunities to explore potentially larger classes of rings.

While proving the existence of an object is one thing, explicitly constructing such an object is another matter. This brings us to an important application of \Cref{thm:strong_generation}. Recall that a Noetherian ring $R$ of prime characteristic is said to be \textit{$F$-finite} if its Frobenius morphism $F\colon R \to R$ is finite. In \cite{BILMP:2023}, it was shown that $F_\ast^e R$ (i.e. the $e$-th iterate of $F\colon R \to R$) is a strong generator for $D^b(\operatorname{mod} R)$. The following statement builds on this result, allowing us to explicitly identify a strong generator for rings of prime characteristic.

\begin{introcor}\label{cor:f_finite_strong_gen_module_cat}
    If $R$ is an $F$-finite Noetherian ring, then there exists $e,c\geq 0$ such that $\bigoplus^c_{i=0} \Omega_R^i (F_\ast^e R)$ is a strong generator for $\operatorname{mod}R$.
\end{introcor}

The only undetermined parameter in \Cref{cor:f_finite_strong_gen_module_cat} is the number of syzygies required, which can be made explicit in the local isolated singularity case.

\begin{introprop}\label{prop:f_finite_isolated singularity_module_category}
    If $R$ is an $F$-finite local ring with isolated singularity of Krull dimension $d$, then, for all $e\gg 0$, $R \oplus (\bigoplus_{i=0}^{d} \Omega_R^i (F^e_\ast R))$ is a strong generator for $\operatorname{mod} R$.    
\end{introprop} 

Our efforts extend a bit further past \Cref{cor:f_finite_strong_gen_module_cat}. Motivated by the proof technique of \Cref{thm:strong_generation}, we establish sharp upper bounds on the Rouquier dimension for $D^b (\operatorname{mod}R)$. Recall that for a Noetherian ring $R$, a chain of ideals $0=I_0\subseteq I_1\subseteq\cdots\subseteq I_n=R$ exists, where each $R$-module $I_i/I_{i-1}$ is isomorphic to $R/\mathfrak{p}_i$ for some prime ideal $\mathfrak{p}_i$
of $R$. See \cite[Theorem 6.4]{Matsumura:1986}. Our strategy, akin to that of \cite[Proposition 3.9]{Aihara/Takahashi:2015}, for bounding the Rouquier dimension of $D^b(\operatorname{mod}R)$ involves  studying the length of such a chain of ideals \Cref{lem:bounds_via_min_primes}. The following result establishes a desirable upper bound on Rouquier dimension in terms of well-established numerical invariants for rings and modules. See \Cref{rmk:bounds} for details.

\begin{introprop}\label{prop:rouq_dim_bounds}
    Let $R$ be a Noetherian ring.
    \begin{enumerate}
        \item If $\operatorname{Ass}(R)=\operatorname{Min}(R)$, then the Rouquier dimension of $D^b(\operatorname{mod} R)$ is bounded above by the following:
        \begin{displaymath}
            \big(\sum_{\mathfrak{p} \in \operatorname{Min}(R)} \ell\ell(R_\mathfrak{p})\big) \underset{\mathfrak{p} \in \operatorname{Min}(R)}{\sup}\{1+\dim D^b(\operatorname{mod} R/\mathfrak{p}) \} -1
        \end{displaymath}
        where $\ell \ell(R_\mathfrak{p})$ denotes Loewy length of the Artinian local ring $R_\mathfrak{p}$.
        \item If $R$ is reduced, then the Rouquier dimension of $D^b (\operatorname{mod}R)$ is bounded above by the following:
        \begin{displaymath}
            |\operatorname{Min}(R)| \underset{\mathfrak{p} \in \operatorname{Min}(R)}{\sup} \{ 1+ \dim D^b (\operatorname{mod}R/\mathfrak{p}) \} -1.
        \end{displaymath}
        \item If $R$ is a Stanley-Reisner ring, then the Rouquier dimension of $D^b (\operatorname{mod}R)$ is bounded above by the following:
        \begin{displaymath}
            |\operatorname{Min}(R)| ( 1 + \dim R) - 1.
        \end{displaymath}
    \end{enumerate}
\end{introprop}

\begin{notation}
    All rings considered are assumed to be commutative and unital. Let $R$ be a Noetherian ring. The category of finitely generated $R$-modules is denoted by $\operatorname{mod}R$, and its associated bounded derived category is denoted $D^b (\operatorname{mod}R)$. The collection of minimal primes of $R$ is denoted $\operatorname{Min}(R)$, and set of associated primes of a module $M$ is denoted $\operatorname{Ass}_R(M)$; the case of $M=R$ we write $\operatorname{Ass}(R)$.
\end{notation}

\begin{ack}
    The authors would like to thank discussions with Srikanth Iyengar,
    Janina Letz, and Josh Pollitz for their wonderfully valuable comments and
    perspectives shared that led to improvements on this manuscript. Additionally, we thank the anonymous referee for helpful suggestions to our work. Souvik Dey
    was partially supported by the Charles University Research Center program
    No.UNCE/SCI/022 and a grant GA \v{C}R 23-05148S from the Czech Science
    Foundation. Ryo Takahashi was partly supported by JSPS Grant-in-Aid for
    Scientific Research 23K03070. Pat Lank was partly supported by National
    Science Foundation under Grant No. DMS-2302263.
\end{ack}

\section{Generation}
\label{sec:generation}

This section briefly recalls background for generation in module categories and its corresponding bounded derived category. The primary sources are \cite{Bondal/VandenBergh:2003, Iyengar/Takahashi:2016, Avramov/Buchweitz/Iyengar/Miller:2010}. Let $R$ be a Noetherian ring.

\subsection{Thick subcategories in \texorpdfstring{$D^b(\operatorname{mod}R)$}{t}}
\label{sec:thick_subcategories_derived_category}

\begin{definition}
    Let $G$ be an object of $D^b (\operatorname{mod}R)$.
    \begin{enumerate}
        \item A full triangulated subcategory $\mathcal{S}$ of $D^b(\operatorname{mod} R)$ is said to be \textbf{thick} when it is closed under direct summands.
        \item The smallest thick subcategory of $D^b(\operatorname{mod}R)$ containing $G$ is denoted $\operatorname{thick} (G)$.
        \item Consider the following additive subcategories of $D^b (\operatorname{mod}R)$:
        \begin{enumerate}
            \item $\operatorname{thick}^0 (G)$ is the full subcategory consisting of all objects isomorphic to the zero object 
            \item $\operatorname{thick}^1 (G)$ is the full subcategory consisting of bounded complexes which are direct summands of objects of the form $\bigoplus_{s\in \mathbb{Z}} G^{\oplus r_s}[s]$ whose differential is zero; i.e. a finite direct sum of shifts of $G$. 
            \item If $n\geq 2$, then $\operatorname{thick}^n (G)$ is the full subcategory consisting of objects which are direct summands of objects $E$ fitting into a distinguished triangle
            \begin{displaymath}
            A \to E \to B \to A[1]
            \end{displaymath}
            where $A$ is in $\operatorname{thick}^{n-1} (G)$ and $B$ is in $\operatorname{thick}^1 (G)$.
        \end{enumerate}
    \end{enumerate}
\end{definition}

\begin{remark}
    This yields an ascending filtration of additive subcategories:
    \begin{displaymath}
        \begin{aligned}
            \operatorname{thick}^0  & (G) \subseteq \operatorname{thick}^1 (G) \subseteq \cdots \\& \subseteq \operatorname{thick}^n (G) \subseteq \cdots \subseteq \bigcup^\infty_{n=0} \operatorname{thick}^n (G)= \operatorname{thick} (G).
        \end{aligned}
    \end{displaymath}
     This filtration keeps track of the minimal number of required cones required to 'build' an object from $G$. 
\end{remark}
    
\begin{definition}
    Let $E,G$ be objects of $D^b (\operatorname{mod}R)$.
    \begin{enumerate}
        \item If $E$ is an object of $\operatorname{thick} (G)$, then we say that $G$ \textbf{finitely builds} $E$.
        \item $G$ is a \textbf{classical generator} for $D^b (\operatorname{mod}R)$ if $\operatorname{thick} (G) = D^b (\operatorname{mod}R)$. 
        \item $G$ is called a \textbf{strong generator} if there exists $n\geq 0$ such that $\operatorname{thick}^{n+1} (G) = D^b (\operatorname{mod}R)$, and the smallest such value $n$ is its \textbf{generation time}. 
        \item The \textbf{Rouquier dimension} of $D^b (\operatorname{mod}R)$ is the minimum over all possible generation times associated to strong generators, and it is denoted $\dim D^b(\operatorname{mod}R)$.
    \end{enumerate}
\end{definition}

\begin{example}
    If $R$ is a quasi-excellent ring of finite Krull dimension, then $D^b (\operatorname{mod}R)$ admit strong generators \cite[Main Theorem]{Aoki:2021}. For instance, if $R$ is additionally regular, then $R$ finitely builds every object $E$ in $D^b (\operatorname{mod}R)$ in at most $\dim R + 1$ cones \cite{Christensen:1998}.
\end{example}

\begin{lemma}\label{lem:reduced_Rouquier_bound}
    Suppose $R$ is a reduced ring. If $D^b(\operatorname{mod}R)$ has finite Rouquier dimension, then $R$ has finite Krull dimension and $\dim D^b(\operatorname{mod}R)$ is at least $\dim R -1$. 
\end{lemma}

\begin{proof} 
    Let $\mathfrak p  \in \operatorname{Spec}(R)$. Then $R_{\mathfrak p}$ is a reduced local ring such that $D^b(\operatorname{mod}R_{\mathfrak p})$ has finite Rouquier dimension \cite[Lemma 4.2(3)]{Aihara/Takahashi:2015}. Recall that Noetherian local rings have finite Krull dimension. Note that $\dim(R_{\mathfrak p})-1\le \dim D^b(\operatorname{mod}R_{\mathfrak p})$ \cite[Corollary 6.6]{Aihara/Takahashi:2015}. Moreover, we know $\dim D^b(\operatorname{mod}R_{\mathfrak p})\le \dim D^b(\operatorname{mod}R)$ via \cite[Lemma 4.2(3)]{Aihara/Takahashi:2015}. Combining the two inequalities, we then get $\dim(R_{\mathfrak p})-1\le \dim D^b(\operatorname{mod}R)$. Since $\mathfrak p \in \operatorname{Spec}(R)$ is arbitrary, we conclude $\sup\{\dim(R_{\mathfrak p})\mid \mathfrak p\in \operatorname{Spec}(R)\}-1\le \dim D^b(\operatorname{mod}R)$. This finishes the proof since $\dim R=\sup\{\dim(R_{\mathfrak p})\mid \mathfrak p\in \operatorname{Spec}(R)\}$. 
\end{proof}  

\subsection{Classical generation in \texorpdfstring{$\operatorname{mod}R$}{}}
\label{sec:thick_subcategories_module_category}

\begin{definition}
    Let $\mathcal{S}$ be a subcategory of $\operatorname{mod}R$.
    \begin{enumerate}
        \item A full additive subcategory of $\operatorname{mod}R$ is \textbf{thick} when it is closed under direct summands and for any short exact sequence
        \begin{displaymath}
            0 \to A \to B \to C \to 0,
        \end{displaymath}
        if two of the three belong to it, then so does the third.
        \item The smallest thick subcategory in $\operatorname{mod}R$ containing $\mathcal{S}$ is denoted by $\operatorname{thick}_{\operatorname{mod}R} (\mathcal{S})$. 
        \item An object $G$ in $\operatorname{mod}R$ is said to be a \textbf{classical generator} when $\operatorname{thick}_{\operatorname{mod}R} (G) = \operatorname{mod}R$. 
    \end{enumerate}
\end{definition}

\begin{example}\label{ex:f_finite_ring}
    Suppose $R$ is a Noetherian ring of prime characteristic $p$ and of finite Krull dimension. Recall the Frobenius morphism $F\colon R \to R$ is given by $r\mapsto r^p$. If $F\colon R \to R$ is a finite ring homomorphism, then $R$ is said to be \textbf{$F$-finite}. That is, $F_\ast R \in \operatorname{mod}R$ where $F_\ast R$ denotes restriction of scalars along $F\colon R \to R$. For instance, any ring that is essentially of finite type over a perfect field of prime characteristic is $F$-finite. These rings are excellent \cite[Theorem 2.5]{Kunz:1976}. In \cite[Corollary 3.9]{BILMP:2023} it was shown there exists $e >0$ such that $F_\ast^e R$ is a strong generator for $D^b (\operatorname{mod}R)$. By \cite[Theorem 1]{Krause/Stevenson:2013} (see also \cite[Theorem 10.5]{Takahashi:2023}), it can be checked that $R\oplus F_\ast^e R$ is a classical generator for $\operatorname{mod}R$.
\end{example}

\begin{definition}
    Suppose $\mathcal{S}$ is a subcategory of $\operatorname{mod}R$. 
    \begin{enumerate}
        \item $\operatorname{add}\mathcal{S}$ is the full subcategory of $\operatorname{mod}R$ consisting of direct summands of finite direct sums of objects in $\mathcal{S}$
        \item $\operatorname{thick}_{\operatorname{mod}R}^0(\mathcal{S})$ is the full subcategory consisting of objects isomorphic to the zero module
        \item $\operatorname{thick}_{\operatorname{mod}R}^1(\mathcal{S}) := \operatorname{add}\mathcal{S}$
        \item $\operatorname{thick}_{\operatorname{mod}R}^n(\mathcal{S})$, for $n\geq 2$, is the full subcategory of $\operatorname{mod}R$ whose objects are direct summands of an object fitting into a short exact sequence
        \begin{displaymath}
            0 \to X \to Y \to Z \to 0
        \end{displaymath}
        where among the other two, one belongs to $\operatorname{thick}_{\operatorname{mod}R}^{n-1}(\mathcal{S})$ and the other in $\operatorname{thick}_{\operatorname{mod}R}^1(\mathcal{S})$. 
    \end{enumerate}
\end{definition}
    
\begin{remark}
    This forms an ascending chain of subcategories in $\operatorname{mod}R$:
    \begin{displaymath}
        \operatorname{thick}_{\operatorname{mod}R}^0 (\mathcal{S}) \subseteq \operatorname{thick}_{\operatorname{mod}R}^1 (\mathcal{S}) \subseteq \operatorname{thick}_{\operatorname{mod}R}^2(\mathcal{S}) \subseteq \cdots \subseteq \operatorname{thick}_{\operatorname{mod}R}(\mathcal{S}).
    \end{displaymath}
\end{remark}

\subsection{Strong generation in \texorpdfstring{$\operatorname{mod}R$}{}}
\label{sec:extension_construction}

\begin{remark}
    Recall a syzygy of an object $M$ in $\operatorname{mod}R$ is the kernel of an epimorphism $\pi \colon P \to M$ where $P$ is a finitely generated projective $R$-module. This is well defined up to stable isomorphism. For any object $M$ in $\operatorname{mod}R$ and $n \geq 0$, an $n$-th syzygy of $M$ is denoted $\Omega^n_R (M)$ and is defined as follows. Set $\Omega^0_R (M):= M$. If $n=1$, then $\Omega^1_R (M)$ is a syzygy of $M$. For $n>1$, then $\Omega^n_R (M)$ is a syzygy of $\Omega^{n-1}_R (M)$.
\end{remark}

\begin{definition}
    Let $\mathcal{S}$ be a collection of objects in $\operatorname{mod}R$.
    \begin{enumerate}
        \item $|\mathcal{S}|_0$ be the full subcategory consisting of objects isomorphic to the zero module 
        \item $|\mathcal{S}|_1 := \operatorname{add}\mathcal{S}$
        \item $|\mathcal{S}|_n$ is the full subcategory of $\operatorname{mod}R$, for $n\geq 2$, whose objects $M$ fit into a short exact sequence
        \begin{displaymath}
            0 \to Y \to M\oplus W \to X \to 0
        \end{displaymath}
        where amongst $X$ and $Y$, one belongs to $|\mathcal{S}|_{n-1}$ and the other in $|\mathcal{S}|_1$.
    \end{enumerate}
\end{definition}

\begin{remark}
    These subcategories form an ascending chain $|\mathcal{S}|_1 \subseteq |\mathcal{S}|_2 \subseteq \cdots$. If $|\mathcal{S}|_a \star |\mathcal{S}|_b$ denotes the full subcategory whose objects $E$ fit into a short exact sequence
    \begin{displaymath}
        0 \to X \to E \oplus E^\prime \to Y \to 0
    \end{displaymath}
    where $X$ is in $|\mathcal{S}|_a$ and $Y$ is in $|\mathcal{S}|_b$, then $|\mathcal{S}|_a \star |\mathcal{S}|_b = |\mathcal{S}|_{a+b}$. For further details, see  \cite[$\S 5$]{Dao/Takahashi:2014}.
\end{remark}

\begin{definition}
    If $G\in \operatorname{mod}R$, then it is a \textbf{strong generator} when $R$ is a direct summand of $G$ and $\Omega^s_R (\operatorname{mod}R)$ is contained in $|G|_n$ for some $s,n\geq 0$ where $\Omega_R^s (\operatorname{mod}R)$ is the collection of objects of the form $\Omega^s_R (M)$ where $M$ is an object of $\operatorname{mod}R$.
\end{definition}

\begin{remark}\label{rmk:thick_categories_module_property}
    If $G\in \operatorname{mod}R$, and $s,n\geq 0$, then
    \begin{enumerate}
        \item $|G|_n \subseteq \operatorname{thick}_{\operatorname{mod}R}^n (G)$;
        \item $\Omega_R^s (\operatorname{mod}R) \subseteq |G|_n \implies \operatorname{thick}_{\operatorname{mod}R}^{s+n} (R\oplus G) = \operatorname{mod}R$.
    \end{enumerate}
    This is \cite[Proposition 4.5 \& Corollary 4.6]{Iyengar/Takahashi:2016}. 
\end{remark}

\begin{example}\label{ex:Iyengar/Takahashi_strong_gen_examples}
    Suppose $R$ has finite Krull dimension $d$.
    \begin{enumerate}
        \item If $R$ is an Artinian ring, then $\operatorname{mod}R \subseteq |R\oplus R/J(R)|_{\ell\ell(R)}$ where $J(R)$ is the Jacobson radical of $R$ and $\ell\ell(R)$ is the Loewy length.
        \item If $R$ is an excellent equicharacteristic local ring, then $\operatorname{mod}R$ admits a strong generator \cite[Theorem 5.3]{Iyengar/Takahashi:2016}.
        \item If $R$ is essentially of finite type over a field, then $\operatorname{mod}R$ admits a strong generator \cite[Theorem 5.4]{Iyengar/Takahashi:2016}.
    \end{enumerate}
\end{example}

\begin{lemma}\label{lem:strong_generation_module_to_derived}
    Suppose $G\in \operatorname{mod}R$.
    \begin{enumerate}
        \item If there exist $n,s\geq 0$ such that $\Omega_R^s (\operatorname{mod}R) \subseteq |G|_n$, then $\operatorname{thick}^{2(n+s)} (G) = D^b (\operatorname{mod}R)$.
        \item If there exists $n\geq 0$ such that $\operatorname{thick}_{\operatorname{mod}R}^n (G)= \operatorname{mod}R$, then $\operatorname{thick}^{2n} (G) = D^b (\operatorname{mod}R)$.
    \end{enumerate}
\end{lemma}

\begin{proof}
    This is \Cref{rmk:thick_categories_module_property} coupled with \cite[Lemma 7.1]{Iyengar/Takahashi:2016}.
\end{proof}

\section{Proofs}
\label{sec:proofs}

This section establishes our main results. We start by recalling a necessary ideal that is related to the existence of strong generators. See \cite[Section 2]{Iyengar/Takahashi:2016} for details.

\begin{definition}(\cite[Definition 2.1]{Iyengar/Takahashi:2016})
    Let $R$ be a Noetherian ring and $n\geq0$. Consider the following ideal of $R$:
    \begin{displaymath}
        \operatorname{ca}^n (R):= \operatorname{ann}_R\operatorname{Ext}^{\geq n}_R (\operatorname{mod}R,\operatorname{mod}R).
    \end{displaymath}
    The \textbf{cohomology annihilator (ideal)} of $R$ is given by $\operatorname{ca}(R) := \cup^\infty_{n=0}\operatorname{ca}^n (R)$.
\end{definition}

\begin{remark}\label{rmk:Rouquier_dim_cohomology_annihilator_nonzero}
    If $R$ is a Noetherian integral domain such that $D^b (\operatorname{mod}R)$ has finite Rouquier dimension $r$, then the $(r+1)$-th cohomology annihilator ideal $\operatorname{ca}^{r+1} (R)$ is nonzero. This is \cite[Theorem 5]{Elagin/Lunts:2018}.
\end{remark}

\begin{remark}\label{rmk:regular_element_syzygy_quotient}
    Let $R$ be a Noetherian ring, and choose an $R$-regular element $x$ of $R$. If $M$ is an object of $\operatorname{mod}R$, then there exists an $R$-module isomorphism for all $n\geq 0$:
    \begin{displaymath}
        \Omega_{R/xR}^n (\Omega^1_R (M)/ x\Omega^1_R (M)) \cong \Omega_R^{n+1} (M)/ x \Omega_R^{n+1} (M).
    \end{displaymath}
    This is \cite[Lemma 5.6]{Dao/Takahashi:2014}.
\end{remark}

\begin{remark}\label{rmk:ses_cohomology_annihilator}
    Let $R$ be a Noetherian ring, and $M$ be an object of $\operatorname{mod}R$. If $a$ is an element of $R$ which annihilates $\operatorname{Ext}_R^1(M,\Omega^1_R (M))$, then there exists a short exact sequence in $\operatorname{mod}R$:
    \begin{displaymath}
        0 \to (0:_M a) \to M \oplus \Omega^1_R (M) \to \Omega_R (M/a M) \to 0.
    \end{displaymath}
    This is \cite[Remark 2.12]{Iyengar/Takahashi:2016}.
\end{remark}

\begin{example}\label{ex:Dieterich_hypersurface_cohomology_annihilator}
    Let $k$ be a field, and consider the ring $R=k \llbracket  x_0,\ldots,x_n \rrbracket  $. If $f$ is in $R$, then the ideal $(\frac{\partial f}{\partial x_0}, \ldots, \frac{\partial f}{\partial x_n})$ is contained in $\operatorname{ca}^{d+1} \big(R/(f) \big)$. This is \cite[Proposition 18]{Dieterich/1987}.
\end{example}

\begin{example}\label{ex:BHST_completion_cohomology_annihilator}
    Let $n\geq 0$. If $(R,\mathfrak{m})$ is a Noetherian local ring of finite Krull dimension, then $\operatorname{ca}^n (\widehat{R})\cap R$ is contained in $\operatorname{ca}^n (R)$ where $\widehat{R}$ denotes the $\mathfrak{m}$-adic completion of $R$. This is \cite[Theorem 4.5.1]{BHST:2016}. 
\end{example}

\begin{remark}\label{rmk:strong_generator_to_derived_category}
    Let $R$ be a Noetherian ring. If there exists an object $G$ of $\operatorname{mod}R$ and $n,s\geq 0$ such that $\Omega^s_R (\operatorname{mod}R)$ is contained in $|G|_n$, then $\operatorname{mod}R = \operatorname{thick}_{\operatorname{mod}R}^{n+s} (R\oplus G)$ and $D^b (\operatorname{mod}R) = \operatorname{thick}^{2(n+s)} (R\oplus G)$. These are respectively \cite[Corollary 4.6]{Iyengar/Takahashi:2016} and \cite[Lemma 7.1]{Iyengar/Takahashi:2016}.
\end{remark}

\begin{lemma}\label{lem:ext_prop}
    If $E$ is in $ |G|_s$, then $|E|_n $ is contained in $|G|_{ns}$ for all $n\geq 1$.
\end{lemma}

\begin{proof}
    This is shown by induction on $n$. There exists a short exact sequence in $\operatorname{mod}R$:
    \begin{displaymath}
        0 \to A \to E \oplus M \to B \to 0
    \end{displaymath}
    where $A $ is in $ |G|_{s-1}$ and $B$ is in $ |G|_1$. By taking direct sums of this short exact sequence, it verifies $(E\oplus M)^{\oplus \alpha}$ in $ |G|_s$ for all $\alpha>0$. Hence, $|E|_1 $ is contained in $|G|_s$, and this establishes the base case. Assume the claim holds for all $1\leq j \leq n$. If $S$ is in $ |E|_n$, then there exists a short exact sequence in $\operatorname{mod}R$:
    \begin{displaymath}
        0 \to X \to S \oplus N \to Y \to 0
    \end{displaymath}
    where $X $ is in $ |E|_n$ and $Y$ is in $ |E|_1$. The induction step ensures $X $ is in $ |G|_{ns}$ and $Y$ is in $ |G|_s$, which implies $S$ is in $ |G|_{(n+1)s}$ as desired.
\end{proof}

\begin{remark}\label{rmk:ses_to_syzygy_ses}
    Let $R$ be a Noetherian ring and $n \geq 0$. If there exists a short exact sequence
    \begin{displaymath}
        0 \to L \to M \to N \to 0
    \end{displaymath}
    in $\operatorname{mod}R$, then there exists a short exact sequence in $\operatorname{mod}R$:
    \begin{displaymath}
        0 \to \Omega^n_R (L) \to \Omega^n_R(M) \to \Omega^n_R(N) \to 0.
    \end{displaymath}
    This is \cite[Remark 2.2]{Iyengar/Takahashi:2016}. Choose an ideal $I$ is contained in $ R$, and let $M$ be an object of $ \operatorname{mod}R/I$. There exists a short exact sequence:
    \begin{displaymath}
        0 \to E \to \Omega_R^1 (M) \to \Omega^1_{R/I}(M) \to 0
    \end{displaymath}
    where $E$ in $ \operatorname{add}(I)$. Furthermore, for each $n\geq 0$, one has that $\Omega_R^n (M)$ is in $|\Omega^n_{R/I} (M) \oplus \big( \oplus^{n-1}_{i=0}\Omega_R^i (I)\big)|_{n+1}$. This is \cite[Proposition 5.3]{Dao/Takahashi:2014}.
\end{remark}

\begin{proposition}\label{prop:strong_generation_via_minimal_primes}
    Let $R$ be a Noetherian ring. If $\operatorname{mod}R/\mathfrak{p}$ admits a strong generator for every prime ideal $\mathfrak{p}$ in $R$, then $\operatorname{mod}R$ admits a strong generator.
\end{proposition}

\begin{proof}
    If $R$ is an integral domain, then there exists nothing to show, so assume it is not so. There exists a chain of ideals:
    \begin{displaymath}
        (0)=I_0 \subseteq I_1 \subseteq \cdots \subseteq I_m = R
    \end{displaymath}
    where $I_{j+1}/I_j \cong R/\mathfrak{p}_j$ for some $\mathfrak{p}_j$ prime ideal in $R$ (see \cite[Theorem 6.4]{Matsumura:1986}). If $R$ is not an integral domain, then $m$ is positive. The hypothesis ensures for each $1\leq j \leq m$ there exists $n_j, s_j\geq 0$ and $G_j $ in $ \operatorname{mod}R/\mathfrak{p}_j$ such that $\Omega^{s_j}_{R/\mathfrak{p}_j} (\operatorname{mod}R/\mathfrak{p}_j)$ is contained in $ |G_j|_{n_j}$. For any object $M$ in $ \operatorname{mod}R$ and $0 \leq j \leq m-1$, there exists a short exact sequence in $\operatorname{mod}R$:
    \begin{displaymath}
        0 \to I_j M \to I_{j+1} M \to I_{j+1}M/I_j M \to 0.
    \end{displaymath}
    If $s:=\max\{s_1,\ldots,s_m\}$, then $\Omega^s_{R/\mathfrak{p}_j} (\operatorname{mod}R/\mathfrak{p}_j)$ is contained in $|\Omega^{s-s_j}_{R/\mathfrak{p}_j} G_j|_{n_j}$ for each $1\leq j \leq m$. Moreover, if $n:=1+ \max\{ n_1,\ldots,n_m\}$, then $|\Omega^{s-s_j}_{R/\mathfrak{p}_j} G_j|_{n_j} $ is contained in $ |\Omega^{s-s_j}_{R/\mathfrak{p}_j} G_j|_n$. Observe there is a uniform choice of $n,s$ as in for each $1\leq j \leq m$, and so $\Omega^s_{R/\mathfrak{p}_j} (\operatorname{mod}R/\mathfrak{p}_j)$ is contained in $  |\Omega^{s-s_j}_{R/\mathfrak{p}_j} G_j|_n$. By abuse of notation, identify $G_j$ with $\Omega^{s-s_j}_{R/\mathfrak{p}_j} G_j$. For some choice of an $s$-th syzygy, there exists a short exact sequence via \Cref{rmk:ses_to_syzygy_ses} in $\operatorname{mod}R$:
    \begin{equation}
        \label{eq:syzygy_ses_module}
        0 \to \Omega_R^s (I_j M) \to \Omega_R^s ( I_{j+1} M) \to \Omega_R^s ( I_{j+1}M/I_j M) \to 0.
    \end{equation}
    By \Cref{rmk:ses_to_syzygy_ses} and \Cref{lem:ext_prop}, it can be verified for each $0\leq j \leq m-1$ that $\Omega^s_R (I_{j+1} M/I_j M)$ belongs to $|\Omega^s_{R/\mathfrak{p}_j}(I_{j+1} M/I_j M) \oplus \big( \bigoplus_{i=0}^{s-1} \Omega_R^i (\mathfrak{p}_j) \big)|_{s+1} $, which itself is contained in $|G_j \oplus \big( \bigoplus_{i=0}^{s-1} \Omega_R^i (\mathfrak{p}_j) \big)|_{n(s+1)}$. 
    
    Another application of \Cref{lem:ext_prop} with an inductive argument tells us $\Omega^s_R (I_{j+1} M)$ belongs to $|\bigoplus^j_{i=0} \big(G_i \oplus (\bigoplus_{l=0}^{s-1} \Omega_R^l (\mathfrak{p}_i)) \big)|_{ n(s+1)(j+1)}$. Indeed, since $I_1 M\in \operatorname{mod}R/\mathfrak{p}_1$, hence $\Omega^s_{R} (I_1 M)\in |G_1 \oplus \big( \bigoplus_{i=0}^{s-1} \Omega_R^i (\mathfrak{p}_1) \big)|_{n(s+1)}$. Similarly, $I_2M/I_1M\in \operatorname{mod} R/\mathfrak{p}_2$ implies $\Omega^s_R (I_2M/I_1M) \in |G_2 \oplus \big( \bigoplus_{i=0}^{s-1} \Omega_R^i (\mathfrak{p}_2) \big)|_{n(s+1)}$. For $i=2$, the short exact sequence in \Cref{eq:syzygy_ses_module} shows $\Omega^n_R (I_2 M)\in |\bigoplus^2_{i=0} \big(G_i \oplus (\bigoplus_{l=0}^{s-1} \Omega_R^l (\mathfrak{p}_i)) \big)_{2n(s+1)}$. If this process continues, then
    \begin{displaymath}
        \Omega^s_R M=\Omega^n_R(I_m M)\in |\bigoplus^j_{i=0} \big(G_i \oplus (\bigoplus_{l=0}^{s-1} \Omega_R^l (\mathfrak{p}_i)) \big)|_{mn(s+1)}.
    \end{displaymath}
    
    Then after working up the ladder of extensions, we see that $\Omega^s_R (\operatorname{mod}R)$ is contained in $|\bigoplus^m_{i=1} \big(G_i \oplus (\bigoplus_{l=0}^{s-1} \Omega_R^l (\mathfrak{p}_i)) \big)|_{ m n (s+1)}$. Therefore, it follows that
    \begin{displaymath}
        R\oplus \bigoplus^m_{i=1} \big(G_i \oplus (\bigoplus_{l=0}^{s-1} \Omega_R^l (\mathfrak{p}_i)) \big)
    \end{displaymath}
    is a strong generator for $\operatorname{mod}R$.
\end{proof} 

\begin{proof}[Proof of \Cref{thm:strong_generation}]
    The last claim regarding ideals of $R$ follows from \Cref{prop:strong_generation_via_minimal_primes} had we known the initial claim holds. Also, the claim about finite Krull dimension of $R$ would follow from \Cref{lem:reduced_Rouquier_bound} and the observations that each $R/\mathfrak p$ is an integral domain and $\dim(R)=\sup\{\dim(R/\mathfrak{p})\mid \mathfrak p\in \operatorname{Min}(R)\}$. So it suffices to prove that the three conditions are equivalent. On one hand, $(2)\implies (3)$, and on the other, \Cref{rmk:Rouquier_dim_cohomology_annihilator_nonzero} ensures $(3)\implies (1)$, so we show $(1)\implies (2)$. This will be done by Noetherian induction on $\operatorname{Spec}(R)$. The base case with the empty scheme is evident. Suppose $Y$ is a closed subscheme of $\operatorname{Spec}(R)$ such that all properly contained closed subschemes of $Y$ satisfy the desired claim that $(1)\implies (2)$. Note that $Y$ corresponds to some quotient ring $R/I$. It suffices to prove the claim in such a case, so without loss of generality, we replace this quotient $R/I$ by $R$ and assume the claim is true for all properly contained closed subschemes in $\operatorname{Spec}(R)$. Additionally, we may impose that $R$ is an integral domain because the claim is a about the quotient rings $R/\mathfrak{p}$ for $\mathfrak{p}\in \operatorname{Spec}(R)$. If $(1)$ holds, then there exists $r+1$ such that $\operatorname{ca}^{r+1}(R) \not= (0)$. Choose an element $a$ in $ \operatorname{ca}^{r+1}(R)$ that is a nonzero non-unit. There is a bijection between the prime ideals of $R/(a)$ and all prime ideals $\mathfrak{p}$ in $R$ such that $a$ in $ \mathfrak{p}$. For any such prime ideal $\mathfrak{p}$ containing $a$, one has that $\operatorname{ca}(R/\mathfrak{p})$ is nonzero, and so $\operatorname{mod}R/\mathfrak{p}$ a admits strong generator via the induction hypothesis as $\operatorname{Spec}(R/(a))$ is a properly contained closed subscheme of $\operatorname{Spec}(R)$ from our choice of $a$. An application of \Cref{prop:strong_generation_via_minimal_primes} ensures there exists an object $G$ in $ \operatorname{mod}R/(a)$ and $n,s\geq 0$ such that $\Omega_{R/(a)}^s (\operatorname{mod}R/(a)) $ is contained in $ |G|_n$. Fix an object $M$ in $ \operatorname{mod}R$. It is enough to show $\Omega_R^{r+s+1} (M) $ is in $ |\Omega^1_R (G)|_n$. Set $N=\Omega^{r+s+1}_R (M)$. If $R$ is an integral domain, then \Cref{rmk:regular_element_syzygy_quotient} ensures there exists an isomorphism:
    \begin{displaymath}
        N/aN = \Omega^{r+s}_{R/aR} (\Omega^1_R (M)/ a \Omega^1_R (M)).
    \end{displaymath}
    Now restriction of scalars ensures $N/aN $ is in $ |G|_n$, and so $\Omega_R^1 (N/aN) $ belongs to $ |\Omega_R^1 (G)|_n$. For each object $L$ in $ \operatorname{mod}R$ there exists an isomorphism
    \begin{displaymath}
        \operatorname{Ext}_R^1( N, L) \cong \operatorname{Ext}^{r+1}_R ( \Omega_R^{s+1} M,L),
    \end{displaymath}
    and so $\operatorname{Ext}_R^1( N, L)$ is annihilated by $a$. Note that $N$ is at least a first syzygy of $M$, and $a$ is a nonzero divisor on $N$ as $R$ is a domain. By \Cref{rmk:ses_cohomology_annihilator}, $N$ is a direct summand of $\Omega_R^1 (N/aN)$. This implies $\Omega_R^{r + s+1} (M) $ in $ |\Omega_R^1 (G)|_n$, and hence, establishes that $R\oplus \Omega^1_R (G)$ is a strong generator for $\operatorname{mod}R$.
\end{proof}

\begin{remark}
    It is worthwhile to note that there is a resemblance of the proof of \Cref{thm:strong_generation} and that of \cite[Theorem 5.1]{Iyengar/Takahashi:2016}. The interesting component is that this strategy is an honest adaptation to the context where the Rouquier dimension of an integral domain is finite, and \Cref{rmk:Rouquier_dim_cohomology_annihilator_nonzero} ensures nonvanishing of the cohomology annihilator ideal.
\end{remark}

\begin{corollary}\label{cor:quasi_excellent_strong_gen}
    For any Noetherian quasi-excellent ring $R$ of finite Krull dimension, $\operatorname{mod}R$ admits a strong generator. 
\end{corollary}

\begin{proof}
    If $R$ is quasi-excellent, then any essentially $R$-algebra of finite type is as well. By \cite[Main Theorem]{Aoki:2021}, it was shown that any Noetherian quasi-excellent of finite Krull dimension has finite Rouquier dimension, and so appealing to \Cref{thm:strong_generation} furnishes the proof of the first claim.
\end{proof}

\begin{proposition}\label{prop:cmiso}
    As a consequence of \Cref{cor:quasi_excellent_strong_gen}, we are abe to deduce the following: Let $R$ be a local  ring whose completion has isolated singularity. Then, $\operatorname{mod}R$ admits a strong generator. If, moreover, $R$ is Cohen--Macaulay, then there exists $H\in \operatorname{mod } R$ such that $\operatorname{CM}(R)=|H|_c$ for some $c>0$. 
\end{proposition}

\begin{proof} 
    Indeed, since the completion $\widehat R$ is local excellent, so by \Cref{cor:quasi_excellent_strong_gen}, there exists $C\in \operatorname{mod } \widehat R$ and integers $n,s>0$ such that $\Omega^n_{\widehat R}(\operatorname{ mod } \widehat R)\subseteq |C|_s$. Then, $\Omega^{d+n}_{\widehat R}(\operatorname{mod} \widehat R)\subseteq |\Omega^d_{\widehat R} C|_s$, where $d:=\dim(R)$. Since $\widehat R$ has isolated singularity, hence remembering every localization of $\widehat R$ has depth at most $d$, Auslander-Buchsbaum formula implies $\Omega^d_{\widehat R} C$ is locally free on punctured spectrum of $\widehat R$. Now by \cite[Corollary 4.4]{BHST:2016} we get that there exists $G\in \operatorname{mod} R$, locally free on punctured spectrum, such that $\Omega^{d+n}_{\widehat R}(\operatorname{mod} \widehat R)\subseteq |\widehat G|_s$. Since the subcategory of modules, locally free on punctured spectrum, is closed under direct summands and extensions, hence an argument similar to the proof of \cite[Theorem 5.11(2)]{Dao/Takahashi:2014} shows that for every $m>0$ and $N\in |\widehat G|_m$, there exists $M\in |G|_m$ such that $N$ is a direct summand of $\widehat M$. Now let $X\in \Omega^{d+n}_{ R}(\operatorname{mod}  R)$, then $\widehat X\in \Omega^{d+n}_{\widehat R}(\operatorname{mod} \widehat R)\subseteq |\widehat G|_s$. Hence, there exists $M\in |G|_s$ such that $\widehat X$ is a direct summand of $\widehat M$. By \cite[Lemma 5.7]{Aihara/Takahashi:2015}, $X\in \operatorname{add}_R(M)\subseteq |G|_s$. Since $X\in \Omega^{d+n}_{ R}(\operatorname{mod}  R)$ was arbitrary, we conclude $\Omega^{d+n}_{ R}(\operatorname{mod}  R)\subseteq |G|_s$, showing $\operatorname{mod} R$ admits a strong generator.

    Now assume $R$ is also Cohen--Macaulay. Then, as already argued, there exists $G\in \operatorname{mod } \widehat R$ and integers $n,s>0$ such that $\Omega^n_{\widehat R}(\operatorname{ mod } \widehat R)\subseteq |G|_s$. An argument similar to the proofs of \cite[Corollary 5.9, Proposition 5.10]{Dao/Takahashi:2014} shows that $\operatorname{CM}(\widehat R)=|\Omega^d_{\widehat R}(G)\oplus W|_{s(d+n+1)}$, where $d:=\dim R$,  $W=\bigoplus_{j=0}^{d+n-1} \Omega^j_{\widehat R}\omega^{\oplus b_j}$, where $\omega$ is the canonical module of $\widehat R$. Since $\widehat R$ has isolated singularity,  the proof of the claim in \cite[Theorem 5.11(2)]{Dao/Takahashi:2014} (and the argument around it) shows that $\operatorname{CM}(R)=|H|_{s(d+n+1)}$ for some $H\in \operatorname{mod }(R)$.  
\end{proof}

Next, we work towards a proof for \Cref{cor:f_finite_strong_gen_module_cat}. Recall there is an ascending chain of additive subcategories:
\begin{displaymath}
    \operatorname{thick}_{\operatorname{mod}R}^0 (G) \subseteq \operatorname{thick}_{\operatorname{mod}R}^1 (G) \subseteq \operatorname{thick}_{\operatorname{mod}R}^2(G) \subseteq \cdots \subseteq \operatorname{thick}_{\operatorname{mod}R}(G).
\end{displaymath}
It is not clear whether or not the following equality holds:
\begin{displaymath}
    \bigcup^\infty_{i=0} \operatorname{thick}_{\operatorname{mod}R}^i (G) = \operatorname{thick}_{\operatorname{mod}R} (G).
\end{displaymath}
To circumvent this, we will need to introduce a new filtration for thick subcategories in module categories.

\begin{definition}
    Let $R$ be a Noetherian ring. Suppose $G$ is an object in $ \operatorname{mod}R$. Consider the following subcategories:
    \begin{enumerate}
        \item $\operatorname{th}_{\operatorname{mod}R}^1 (G):=\operatorname{add} (G)$ 
        \item $\operatorname{th}_{\operatorname{mod}R}^n (G)$ is the full subcategory of $\operatorname{mod} R$ consisting of modules $M$ such that there exists an exact sequence in $\operatorname{mod}R$:
        \begin{displaymath}
            0 \to A \to B \to C \to 0
        \end{displaymath} 
        where $D_1$ and $D_2$ belong to $\operatorname{th}_{\operatorname{mod}R}^{n-1} (G)$ and $M$ is a direct summand of $D_3$ for $\{D_1,D_2,D_3\}=\{A,B,C\}$.
    \end{enumerate}
\end{definition}

\begin{lemma}\label{lem:large_thickenings_filtration}
    If $G$ in $ \operatorname{mod}R$, then $\bigcup_{n=0}^\infty \operatorname{th}_{\operatorname{mod}R}^n (G) = \operatorname{thick}_{\operatorname{mod}R} (G)$.
\end{lemma}

\begin{proof}
    First, one shows that $\operatorname{th}_{\operatorname{mod}R}^n (G) $ is contained in $ \operatorname{thick}_{\operatorname{mod}R} (G)$ for each $n\geq 0$. This may be done by induction on $n$. To show the reverse inclusion, it suffices to check that $\bigcup_{n=0}^\infty \operatorname{th}_{\operatorname{mod}R}^n (G)$ is thick subcategory in $\operatorname{mod}R$.
\end{proof}

\begin{lemma}\label{lem:large_thickening_to_extension_decompositions}
    Let $G,M$ in $ \operatorname{mod}R$, and choose $n\geq 0$. If $M$ belongs to $\operatorname{th}_{\operatorname{mod}R}^{n+1}(G)$, then $\Omega_R^n (M)$ belongs to $|C|_m$ where $C=R\oplus\bigoplus_{i=0}^{2n}\Omega_R^i (G)$ and $m=2^n$.
\end{lemma}

\begin{proof}
    We use induction on $n$. The assertion is clear if $n=0$. Let $n>0$.
    There is an exact sequence
    \begin{displaymath}
        0 \to A \to B \to C \to 0
    \end{displaymath}
    such that $D_1$ and $D_2$ belong to $\operatorname{th}_{\operatorname{mod}R}^n (G)$ and $M$ is a direct summand of $D_3$
    where $\{D_1,D_2,D_3\}=\{A,B,C\}$. The induction hypothesis implies that
    $\Omega_R^{n-1}(D_1)$ and $\Omega_R^{n-1}(D_2)$ are in $|E|_m$, where
    $E=R\oplus\bigoplus_{i=0}^{2n-2}\Omega_R^i (G)$ and $m=2^{n-1}$.

    (1) Suppose $(D_1,D_2,D_3)=(C,A,B)$. Then there is an exact sequence
    \begin{displaymath}
        0 \to \Omega_R^{n-1}(D_2) \to \Omega_R^{n-1}(D_3) \to \Omega_R^{n-1}(D_1) \to 0
    \end{displaymath}
    up to projective summands. Hence, $\Omega_R^{n-1}(M)$ belongs to $|E|_{2m}$.
    Therefore, $\Omega_R^n (M)$ belongs to $|\Omega_R (E)|_{2m}$.

    (2) Suppose $(D_1,D_2,D_3)=(B,C,A)$. Then there is an exact sequence
    \begin{displaymath}
        0 \to \Omega_R^{n-1}(D_3) \to \Omega_R^{n-1}(D_1) \to \Omega_R^{n-1}(D_2) \to 0
    \end{displaymath}
    up to projective summands, which induces an exact sequence
    \begin{displaymath}
        0 \to \Omega_R^n (D_2) \to \Omega_R^{n-1}(D_3) \to \Omega_R^{n-1}(D_1) \to 0
    \end{displaymath}
    up to projective summands. As $\Omega_R^{n-1} (D_2)$ is in $|E|_m$, we have
    $\Omega^n_R (D_2)$ is in $|\Omega_R^1 (E)|_m$. Hence, $\Omega_R^{n-1}(M)$ belongs to $|E\oplus\Omega_R^1 (E)|_{2m}$. Therefore, $\Omega^n_R (M)$ belongs to $|\Omega_R^1 (E)\oplus\Omega_R^2(E)|_{2m}$.

    (3) Suppose $(D_1,D_2,D_3)=(A,B,C)$. Then there is an exact sequence
    \begin{displaymath}
        0 \to \Omega_R^{n-1} (D_1) \to \Omega_R^{n-1} (D_2) \to \Omega_R^{n-1} (D_3) \to 0
    \end{displaymath}
    up to projective summands, which induces an exact sequence
    \begin{displaymath}
        0 \to \Omega_R^n (D_3) \to \Omega_R^{n-1} (D_1) \to \Omega_R^{n-1} (D_2) \to 0
    \end{displaymath}
    up to projective summands. However, this further induces an exact sequence
    \begin{displaymath}
        0 \to \Omega_R^n (D_2) \to \Omega_R^n (D_3) \to \Omega_R^{n-1} (D_1) \to 0
    \end{displaymath}
    up to projective summands. As $\Omega_R^{n-1} (D_2)$ is in $|E|_m$, we have
    $\Omega_R^n (D_2)$ is in $|\Omega_R^1 (E)|_m$. Hence, $\Omega_R^n (M)$ belongs to $|E\oplus\Omega_R (E)|_{2m}$.

    Consequently, $\Omega_R^n (M)$ belongs to $|E\oplus\Omega_R^1
    (E)\oplus\Omega_R^2(E)|_{2m}=|F|_{2m}$, where $F=R\oplus\bigoplus_{i=0}^{2n}\Omega_R^i (G)$ and $2m=2^n$.
\end{proof}

\begin{proposition}\label{prop:large_thickening_to_syzygy}
    Let $G,M$ in $ \operatorname{mod}R$.
    \begin{enumerate}
        \item For $n=1,2$ one has $\operatorname{thick}_{\operatorname{mod}R}^n (G)=\operatorname{th}_{\operatorname{mod}R}^n (G)$, and $\operatorname{thick}_{\operatorname{mod}R}^n (G) $ is contained in $ \operatorname{th}_{\operatorname{mod}R}^n (G)$ when $n>2$.
        \item Let $n\geq 0$ be an integer. Put $m=2^n$ and $C=R\oplus(\bigoplus_{i=0}^{2n}\Omega_R^i (G))$. Consider the following three conditions.
        \begin{enumerate}
            \item The module $M$ belongs to $\operatorname{th}_{\operatorname{mod}R}^{n+1} (G)$.
            \item The module $\Omega_R^n (M)$ belongs to $|C|_m$.
            \item The module $M$ belongs to $\operatorname{thick}^{m+n}_{\operatorname{mod}R}(R\oplus C)$.
        \end{enumerate}
        Then the implications $(a) \implies (b) \implies (c)$ hold.
    \end{enumerate}
\end{proposition}

\begin{proof}
    (1) The claim is immediate from the definitions of $\operatorname{thick}_{\operatorname{mod}R}^n (G)$ and $\operatorname{th}_{\operatorname{mod}R}^n (G)$.

    (2) \Cref{lem:large_thickening_to_extension_decompositions} ensures $(a)$ implies $(b)$. Let us show that $(b)$ implies $(c)$. Assume $\Omega_R^n (M)$ belongs to $|C|_m$. There is an exact sequence 
    \begin{displaymath}
    0 \to \Omega_R^n (M) \to P \to \Omega_R^{n-1} (M) \to 0.
    \end{displaymath}
    where $P$ is a projective $R$-module. We have that $\Omega_R^n (M)$ is in $\operatorname{thick}_{\operatorname{mod}R}^m(R\oplus C)$ and $R^{\oplus r}$ is in $\operatorname{add}(R\oplus C)$. Hence, $\Omega_R^{n-1} (M)$ is in $\operatorname{thick}_{\operatorname{mod}R}^{m+1}(R\oplus C)$. It may be verified that $M$ in $ \operatorname{thick}_{\operatorname{mod}R}^{m+n}(R\oplus C)$ by induction on $n$.
\end{proof}

\begin{corollary}
    The following are equivalent:
    \begin{enumerate}
        \item $\operatorname{mod} R=\operatorname{th}_{\operatorname{mod}R}^n(G)$ for some $G$ in $ \operatorname{mod}R$ and $n\geq 0$;
        \item $\operatorname{mod} R=\operatorname{thick}_{\operatorname{mod}R}^m (C)$ for some $C$ in $ \operatorname{mod}R$ and $m\geq 0$.
    \end{enumerate}
\end{corollary}

\begin{proof}
    $(1)\implies (2)$ by \Cref{prop:large_thickening_to_syzygy}, and $(2)\implies (1)$ comes by construction.
\end{proof}

\begin{proof}[Proof of \Cref{cor:f_finite_strong_gen_module_cat}]
    By \Cref{ex:f_finite_ring}, we know that $R\oplus F_\ast^e R$ is a classical generator for $\operatorname{mod}R$. However, \Cref{cor:quasi_excellent_strong_gen} promises that there exist $G$ in $ \operatorname{mod}R$ and $n,s \geq 0$ such that $\Omega^s_R (\operatorname{mod}R) $ is contained in $ |G|_n$. \Cref{lem:large_thickenings_filtration} ensures there exists $l\geq 0$ such that $G$ in $\operatorname{th}_{\operatorname{mod}R}^l (R\oplus F_\ast^e R)$. Appealing to \Cref{prop:large_thickening_to_syzygy}, we see that there exists $c \geq 0$ such that $\Omega^l_R (G)$ is an object of $| R \oplus \bigoplus^c_{i=0} \Omega^i_R (F_\ast^e R) |_{l-1}$. Note that the short exact sequence
    \begin{displaymath}
        0 \to \Omega^1_R (F_\ast^e R) \to R^{\oplus s} \to F_\ast^e R \to 0
    \end{displaymath}
    tells us that $R$ belongs to $|\bigoplus^c_{i=0} \Omega^i_R (F_\ast^e R)|_2$. Hence, taking a few more extensions if needed, we can see that $\Omega^l_R (G)$ is an object of $|\bigoplus^c_{i=0} \Omega^i_R (F_\ast^e R) |_L$. By taking a few more syzygies, it follows that $\Omega^{l+s}_R (\operatorname{mod}R)$ is contained in $|\bigoplus^{s+c}_{i=0} \Omega^i_R (F_\ast^e R) |_N$ for some $N \gg 0$. These detail last details come from the fact that $\Omega^l_R (G)$ is an object of $| R \oplus \bigoplus^c_{i=0} \Omega^i_R (F_\ast^e R) |_{l-1}$ and $\Omega^\alpha_R (\operatorname{mod}R)\subseteq |\Omega^\beta_R (G) |_\gamma$ for $\alpha,\beta,\gamma\gg 0$.
\end{proof}

Next we prove  \Cref{prop:f_finite_isolated singularity_module_category}, for which we first need a lemma. 

\begin{lemma}\label{lem:isolated_singularity_lemma}
     Let $R$ be a Noetherian ring. The following hold true 
    \begin{enumerate}
        \item Let $n\geq 0$, and suppose that $0\to M \to C_0\to \cdots\to C_{n-1}\to N\to 0$ is an exact sequence in $\operatorname{mod} R$. Then 
        \begin{displaymath}
            \Omega^n_R (N)\in |M\oplus (\oplus_{i=0}^{n-1}\Omega_R^{i+1}(C_i))|_{n+1}.
        \end{displaymath}
    
        \item If $x_1,\ldots,x_n$ be an $M$-regular sequence in $R$, then 
        \begin{displaymath}
            \Omega^n_R(M/\mathbf x M)\in |\oplus_{i=0}^n \Omega^i_R M|_{n+1}.
        \end{displaymath}
    
        \item If $(R,\mathfrak m,k)$ is an $F$-finite local ring of depth $t$, then 
        \begin{displaymath}
            \Omega^t_R (k) \in |R \oplus (\oplus_{i=0}^t \Omega^i_R (F^e_* R))|_{t+1}
        \end{displaymath}
        for all $e\gg 0$.    
    \end{enumerate}
\end{lemma}

\begin{proof}
    (1) Let us induct on $n$, and note the case where $n=0$ is obvious. If $n=1$, then the short exact sequence
    \begin{displaymath}
        0 \to M \to C_0 \to N \to 0
    \end{displaymath}
    yields another short exact sequence
    \begin{displaymath}
        0\to \Omega_R (C_0) \to \Omega_R (N) \to M\to 0    
    \end{displaymath}
    by repeatedly applying \cite[Proposition 2.2(1)]{Dao/Takahashi:2015b}. Hence, the object $\Omega_R (N)$ belongs to $|M\oplus \Omega_R (C_0)|_2$. 
    
    Assume $n\geq 2$. There are exact sequences 
    \begin{displaymath}
        \begin{aligned}
            & 0\to M \to C_0\to \cdots \to C_{n-2}\to L\to 0,
            \\& 0\to L \to C_{n-1}\to N\to 0.
        \end{aligned}
    \end{displaymath}
    The first exact sequence, in view of induction hypothesis, tells us that $\Omega_R^{n-1} (L)\in |M\oplus(\oplus_{i=0}^{n-2} \Omega_R^{i+1}(C_i))|_n$. The short exact sequence 
    \begin{displaymath}
        0\to L \to C_{n-1}\to N\to 0
    \end{displaymath}
    yields another short exact sequence,
    \begin{displaymath}
        0\to \Omega_R (C_{n-1})\to \Omega_R (N) \to L \to 0    
    \end{displaymath}
    by repeatedly applying \cite[Proposition 2.2(1)]{Dao/Takahashi:2015b}, and which in turn provides 
    \begin{displaymath}
        0\to \Omega^n_R (C_{n-1}) \to \Omega^n_R (N) \to \Omega^{n-1}_R  (L) \to 0    
    \end{displaymath}
    by \cite[Proposition 2.2(1)]{Dao/Takahashi:2015b} once more. Thus, $\Omega_R^n (N)$ is in $|M\oplus \Omega_R^n (C_{n-1})\oplus(\oplus_{i=0}^{n-2} \Omega_R^{i+1}(C_i))|_{n+1}$, finishing the inductive step. 

    (2) As $\mathbf x :=x_1,\ldots,x_n$ is $M$-regular, so we have an exact sequence arising from Koszul complex
    \begin{displaymath}
        0\to M \to L_0\to \cdots \to L_{n-1}\to M/\mathbf x M \to 0
    \end{displaymath}
    where each $L_i$ is a finite direct sum of copies of $M$. Hence, $\Omega^n_R(M/\mathbf x M)$ belongs to $|\oplus_{i=0}^n \Omega^i_R (M)|_{n+1}$ by part (1).  

    (3) By \cite[Corollary 3.3]{Takahashi/Yoshino:2004}, for every $e\gg 0$, there exists an $F^e_\ast R$-regular sequence $\mathbf x :=x_1,\cdots,x_t$ such that $k$ is a direct summand of $F^e_\ast R/(\mathbf x)$. Hence, $\Omega^t_R (k) \in |R \oplus (\oplus_{i=0}^t \Omega^i_R (F^e_\ast R))|_{t+1}$ by part (2) and remembering that syzygies are defined up to free summand. 
\end{proof} 

\begin{proof}[Proof of \Cref{prop:f_finite_isolated singularity_module_category}] Set $t=\operatorname{depth } R$ and $k$ to be the residue field of $R$.
As $R$ is $F$-finite, it is excellent, and hence $\operatorname{mod} R$ has a strong generator by \Cref{cor:quasi_excellent_strong_gen}. That is, $\Omega_R^n(\operatorname{mod} R) \subseteq |G|_a$  for some $G\in \operatorname{mod} R$ and integers $n,a>0$. By \cite[Theorem 3.2]{BHST:2016}, we get $\Omega_R^n(\operatorname{mod} R)\subseteq|\oplus_{i=0}^d\Omega^i_R (k)|_b$ for some $b\geq 1$. So, $\Omega^{n+t}(\operatorname{mod} R)\subseteq|\oplus_{i=0}^d\Omega^{i+t}_R (k)|_b$. Now, $\Omega_R^{i+t}(k)$ is locally free on punctured spectrum, and has depth at least $t$ for each $i\geq 0$. Thus, $G:=\oplus_{i=0}^d\Omega^{i+t}_R (k) \in |\oplus_{i=t}^d \Omega_R^i(k)|$ by \cite[Theorem 4.1]{BHST:2016}. Hence by \cite[Remark 2.11]{BHST:2016} we have $G\in |\oplus_{i=t}^d \Omega_R^i(k)|_v$ for some $v$. Thus, $\Omega^{n+t}(\operatorname{mod} R)\subseteq|G|_b\subseteq |\oplus_{i=t}^d \Omega_R^i(k)|_{bv}$. Since for all $e\gg 0$ we have $\Omega^t_R (k) \in |R \oplus (\oplus_{i=0}^t \Omega^i_R (F^e_\ast R))|_{t+1}$ by \Cref{lem:isolated_singularity_lemma}, so $\Omega_R^j(k)\in |R \oplus (\oplus_{i=j-t}^j \Omega^i_R (F^e_\ast R))|_{t+1}$ for all $j\geq t$. Thus, $\Omega^{n+t}(\operatorname{mod} R)\subseteq |\oplus_{i=t}^d \Omega_R^i(k)|_{bv}\subseteq |R\oplus(\oplus_{i=0}^d \Omega^i_R(F^e_\ast R))|_{r}$ for some $r$.  
\end{proof}

\begin{remark}
     Recall a module $M$ is said to satisfy \textit{Serre's condition $(S_n)$}, where $n\geq 0$, if for all prime ideals $\mathfrak{p}$ of $R$:
    \begin{displaymath}
        \operatorname{depth}_{R_{\mathfrak{p}}} M_{\mathfrak{p}} \geq \inf\{n,\dim R_{\mathfrak{p}}\}.
    \end{displaymath}
    It is clear that $M$ satisfies $(S_1)$ if, and only if, $\operatorname{Ass}_R(M)$ is contained in $\operatorname{Min}(R)$. For a Noetherian local ring $(R,\mathfrak{m})$ and $M$ an object of $\operatorname{mod} R$, the \textit{Loewy length of $M$} is given by:
    \begin{displaymath}
        \ell\ell(M):=\inf\{n\geq 0: \mathfrak{m}^n M=0\}.    
    \end{displaymath}
    Observe that $\ell\ell(M)<\infty$ if, and only if, $M$ has finite length $\ell(M)$. Moreover, it is always the case that $\ell\ell(M)\leq\ell(M)$. 
\end{remark}

\begin{lemma}\label{lem:bounds_via_min_primes}
    Suppose $R$ is a Noetherian ring and $M$ is an object of $\operatorname{mod}R$.
    \begin{enumerate}
        \item If $\mathfrak{p}\in \operatorname{Ass}_R(M)$, then there exists an $R$-submodule $N$ of $M$ such that $\operatorname{Ass}_R(N)=\operatorname{Ass}_R(M)\setminus \{\mathfrak p\}$ and $\operatorname{Ass}_R(M/N)=\{\mathfrak p\}$.


        \item If $M$ satisfies $(S_1)$, then there exists a filtration 
        \begin{displaymath}
            0=M_0\subseteq M_1\subseteq \cdots \subseteq M_l=M,
        \end{displaymath}
        where $l\leq \sum_{\mathfrak{p}\in \operatorname{Ass}_R(M)} \ell\ell(R_{\mathfrak{p}})$, and for every $i$, there exists $\mathfrak{p}_i$ in $\operatorname{Min}(R)$ such that $M_i/M_{i-1}$ belongs to $\operatorname{mod}R/\mathfrak{p}_i$.  
        \item If $R$ is local and $M$ satisfies $(S_1)$, then a filtration as in (2) can be chosen such that $M_i/M_{i-1}$ has positive depth for each $i$.
    \end{enumerate}  
\end{lemma}

\begin{proof}
    (1) This follows from \cite[Lemma 3.5]{iIima/Matsui/Shimada/Takahashi:2022}.


    (2) This will be shown by induction on the cardinality $|\operatorname{Ass}_R(M)|$. Since $M= (0)$ is equivalent to $|\operatorname{Ass}_R(M)|=0$, the case where $|\operatorname{Ass}_R(M)|=1$ may be considered, say $\operatorname{Ass}_R(M)=\{\mathfrak{q}\}$. As $M$ satisfies $(S_1)$, $\mathfrak{q}$ belonging to $ \operatorname{Min}(R)$ ensures $R_{\mathfrak{q}}$ is Artinian local ring. Set $l:=\ell\ell(R_{\mathfrak{q}})$. First, it is checked $\mathfrak{q}^l M=(0)$. Indeed, if not, then the chain of inclusions
    \begin{displaymath}
        \emptyset \neq \operatorname{Ass}_R(\mathfrak{q}^l M)\subseteq \operatorname{Ass}_R(M)=\{\mathfrak{q}\}    
    \end{displaymath}
    implies $\{\mathfrak{q}\}=\operatorname{Ass}_R(\mathfrak{q}^l M)$, and so this guarantees $\mathfrak{q}^lM_{\mathfrak{q}}\neq (0)$, contradicting $\mathfrak{q}^lR_{\mathfrak{q}}=(0)$. Hence, this exhibits $\mathfrak{q}^l M=(0)$. There is the desired filtration, 
    \begin{displaymath}
        0=\mathfrak{q}^l M\subseteq \mathfrak{q}^{l-1}M\subseteq \cdots \subseteq \mathfrak{q}^0M=M.   
    \end{displaymath}
    Assume that $\operatorname{Ass}_R(M)=\{\mathfrak{q}_1,\cdots,\mathfrak{q}_t\}$. Since $M$ satisfies $(S_1)$, each $\mathfrak{q}_i$ is in $\operatorname{Min}(R)$.  By$(1)$, there exists a submodule $N$ of $M$ and  two chains of inclusions
    \begin{displaymath}
        \begin{aligned}
            &\operatorname{Ass}_R(N)\subseteq \{\mathfrak{q}_2,\cdots,\mathfrak{q}_t\}\subseteq \operatorname{Min}(R),
            \\& \operatorname{Ass}_R(M/N)\subseteq \{\mathfrak{q}_1\}\subseteq \operatorname{Min}(R).
        \end{aligned}
    \end{displaymath}
    Hence, both $N$ and $M/N$ satisfy $(S_1)$. By the induction hypothesis, one has filtrations 
    \begin{displaymath}
         0=N_0\subseteq N_1\subseteq \cdots \subseteq N_l=N
    \end{displaymath}
    and
    \begin{displaymath}
        0=\dfrac{M_0^\prime}{N}\subseteq \dfrac{M_1^\prime}{N}\subseteq \cdots \subseteq \dfrac{M_j^\prime}{N}=M/N,
    \end{displaymath}
    where
    \begin{enumerate}
        \item $l\leq \sum_{\mathfrak{q}\in \operatorname{Ass}_R(N)} \ell\ell(R_{\mathfrak{q}}) \leq \sum_{i=2}^t\ell\ell(R_{\mathfrak{q}_i})$,
        \item $j\leq \sum_{\mathfrak{q} \in \operatorname{Ass}_R(M/N)}\ell\ell(R_{\mathfrak{q}})\leq \ell\ell(R_{\mathfrak{q}_1})$.
    \end{enumerate}
    Furthermore, for every $i$ there exists $\mathfrak{p}_i$ in $\operatorname{Min}(R)$ such that $N_i/N_{i-1}$ is in $\operatorname{mod}R/\mathfrak{p}_i$, and similarly 
    \begin{displaymath}
        M^\prime_i/M^\prime_{i-1}\cong \dfrac{M_i^\prime/N}{M^\prime_{i-1}/N}\in \operatorname{mod}R/\mathfrak{p}_i^\prime
    \end{displaymath}
    where $\mathfrak{p}^\prime_i$ is in $\operatorname{Min}(R)$. There is a filtration with required factor modules and has $l+j+1$ many submodules in-between
    \begin{displaymath}
        0=N_0\subseteq N_1\subseteq \cdots \subseteq N_l=N=M^\prime_0\subseteq M^\prime_1\subseteq \cdots\subseteq M^\prime_j=M    
    \end{displaymath}
    satisfying
    \begin{displaymath}
        l+j\leq \sum_{i=2}^t\ell\ell(R_{\mathfrak{q}_i})+\ell\ell(R_{\mathfrak{q}_1})=\sum_{\mathfrak{p}\in \operatorname{Ass}_R(M)} \ell\ell(R_{\mathfrak{p}}).
    \end{displaymath}   

    (3) This follows by applying the relevant part of the proof of \cite[Theorem 5.1]{Kawasaki/Nakumara/Shimada:2019} and using induction similar to (2).   
\end{proof}

\begin{proof}[Proof of \Cref{prop:rouq_dim_bounds}]
     Observe that $(2)$ and $(3)$ are easily deduced from $(1)$. To prove $(1)$, an application of \Cref{lem:bounds_via_min_primes} to the case $M=R$ gives us a filtration, by ideals $I_i$ in $R$, of length at most $\sum_{\mathfrak{p} \in \operatorname{Min}(R)} \ell\ell(R_\mathfrak{p})$, such that each $I_i/I_{i-1}$ is an $R/\mathfrak{p}_i$-module, i.e. annihilated by $\mathfrak{p}_i$ as an $R$-module, for some $\mathfrak{p}_i$ in $\operatorname{Min}(R)$. Hence, for every object $X$ in $D^b(\operatorname{mod}R)$, we can regard each $I_i X/I_{i-1} X$ to be in $D^b( \operatorname{mod} R/\mathfrak{p}_i)$. The remaining of the proof follows verbatim of \cite[Proposition 3.9]{Aihara/Takahashi:2015}.
\end{proof}

\begin{remark}\label{rmk:bounds}
    \Cref{prop:rouq_dim_bounds} is a strengthening to \cite[Corollary 3.11]{Aihara/Takahashi:2015} in the case of Stanley-Reisner rings. Note that the case of a non-hypersurface Artinian local ring $R$, we have that the Rouquier dimension of $D^b (\operatorname{mod}R)$ is nonzero (otherwise $R$ has finite representation type); if $D^b(\operatorname{mod} R)= \langle G \rangle_1$ for some object $G$, then $\operatorname{mod} R= \operatorname{add}(\bigoplus_n H^n (G)$). However, the second claim of \Cref{prop:rouq_dim_bounds} is zero, which exhibits obstacles that arise if we only assume $\operatorname{Ass}(R)=\operatorname{Min}(R)$. Additionally, the first claim of \Cref{prop:rouq_dim_bounds} in the case of an Artinian local ring gives us the familiar upper bound of $\ell\ell(R)-1$ where $\ell\ell(R)$ is the Loewy length.
\end{remark}

\bibliographystyle{alpha}
\bibliography{mainbib}

\end{document}